\documentclass[12pt,a4paper,twoside]{amsart}
\usepackage{amssymb}
\usepackage{amscd} 

\theoremstyle{plain}
 \newtheorem{thm}{Theorem}[section]
 \newtheorem{prop}{Proposition}[section]

\theoremstyle{definition}
 
 \newtheorem{rem}{Remark}[section]
 
\numberwithin{equation}{section}

\renewcommand{\le}{\leqslant}\renewcommand{\leq}{\leqslant}
\renewcommand{\ge}{\geqslant}

\title[Running title]{An application of a moment problem to completely monotonic functions}



\author[Jovanovi\'c]{\bfseries Vladimir Jovanovi\'c $^1$}
\author[Treml]{Milanka Treml $^1$}
\address{$^1$
Faculty of Sciences and Mathematics \\ 
University of Banja Luka  \\ 
Banja Luka\\
Republic of Srpska, Bosnia and Herzegovina}
\email{vladimir.jovanovic@pmf.unibl.org}
\email{milanka.treml@pmf.unibl.org}

\begin{document}

\begin{abstract}
We consider the following question: if a function of the form $\int_0^{\infty}\varphi(t)\, e^{-xt}dt$ is completely monotonic, is it then $\varphi\ge0$? It turns out that the question is related to a moment problem. In the end we apply those results to answer some questions concerning complete monotonicity of certain functions raised in F.~Qi and R.~Agarwal, \textit{On complete monotonicity for several classes of   functions related to ratios of gamma functions}, J Inequal Appl (2019).
\end{abstract}

\maketitle

\section{Introduction}  
At the beginning we review basic notions and facts related to \textit{completely monotonic functions}. An infinitely differentiable function $f:(0,\infty)\rightarrow\mathbb{R}$ is called \textit{completely monotonic}, if
$$(-1)^nf^{(n)}\ge0,\quad n=0,1,2,\dots.$$
The crucial fact concerning this class of functions is \textit{Bernstein theorem}: a function $f$ is completely monotonic is and only if there exists a positive Borel measure $\mu$ on $[0,\infty)$, such that
\begin{equation}\label{eq:Bern}
f(x)=\int_{[0,\infty)}e^{-xt}d\mu(t),
\end{equation}
for all $x>0$. Furthermore, the measure $\mu$ is uniquely determined (see \cite{BergForst}, p.\, 61). In many applications one comes up to the situation that a function of the form $\int_0^{\infty}\varphi(t)\, e^{-xt}dt$ is completely monotonic. Usually, in view of Bernstein theorem, it is tacitly assumed that the function $\varphi$ is then necessarily non-negative. Our aim here is to clarify this question: we give a sufficient condition on $\varphi$ which guarantees the claim and provide a complete proof. It turns out that our question has to do with  uniqueness of measures in a moment problem which we consider in the next section. In the sequel we apply those results in order to answer the question (in a slightly more general form) raised in \cite{QiAgar} on page 34 whether the functions $\psi'(x+1)-\sinh\frac1{x+1}$ and $\frac12\sinh\frac2x - \psi'(x+1)$ are completely monotonic, where $\psi$ is digamma function.
\section{A moment problem}
As we previously mentioned, our considerations are tightly related to the uniqueness question for measures in a moment problem, which we state in Theorem \ref{moment} (see below). It resembles the \textit{Stieltjes moment problem}: if two non-negative measures $\mu$ and $\nu$ with support on $[0,\infty)$ have the same moments, that is, if
$\int_0^{\infty}t^n\, d\mu(t)=\int_0^{\infty}t^n\, d\nu(t)$, for all $n=0,1,\dots$, is it then $\mu=\nu$? In our case, we use a substitution and reduce it to the \textit{Hausdorff moment problem}, where the support of measures is $[0,1]$.
Now, we turn to our moment problem.
\begin{thm}\label{moment}
Assume $\mu$ and $\nu$ are complex Borel measures on $[0,\infty)$ with the property
\[
\int_{[0,\infty)}e^{-nt}\, d\mu(t)=\int_{[0,\infty)}e^{-nt}\, d\mu(t), \enspace n=0,1,2\dots.
\]
Then, $\mu=\nu$.
\end{thm}
We need the following change of variables formula.
\begin{prop}\label{change}
Let $(X,\mathcal{M},\mu)$ be a measure space, $(Y,\mathcal{N})$ a measurable space and $F:X\to Y$ a measurable map. Then for every measurable function $f:Y\to \mathbb{C}$ and every $E\in\mathcal{N}$ we have
\[
\displaystyle \int_E f(y)\, dF_*\mu(y)=\int_{F^{-1}(E)} f(F(x))\, d\mu(x),
\]
in the case either of two sides is defined. Here $F_*\mu=\mu\circ F^{-1}$ is a measure on  $(Y,\mathcal{N})$, the so-called push-forward of $\mu$. 
\end{prop}
\noindent For the proof, see \cite[p.\, 30-31]{Durrett}.
\noindent\begin{rem}\label{Remark}
We notice that the change of variable formula also holds for complex Borel measures.
\end{rem}
\textit{Proof of Theorem \ref{moment}}.\\[1ex]
Recall that the complex measures $\mu$ and $\nu$ are of bounded variation, $M_{\mu}:=|\mu|([0,\infty))<\infty$ and $ M_{\nu}:=|\nu|([0,\infty))<\infty$ (see \cite{Rudin}). Let us define a homeomorphism $F:[0,\infty)\rightarrow (0,1]$, $F(t)=e^{-t}$. Applying Proposition \ref{change} (more precisely Remark \ref{Remark}), we obtain
\[
\displaystyle \int_{[0,\infty)}e^{-nt}d\mu(t)=\int_{(0,1]}s^n\, dF_*\mu(s),\quad \int_{[0,\infty)}e^{-nt}d\nu(t)=\int_{(0,1]}s^n\, dF_*\nu(s).
\]
From the assumptions of Theorem \ref{moment}, we have
\[
\int_{(0,1]}s^n\, dF_*\mu(s)=\int_{(0,1]}s^n\, dF_*\nu(s),
\]
for all $n=0,1,2\dots$. Hence
\begin{equation}
\int_{(0,1]}P(s) \, dF_*\mu(s)=\int_{(0,1]}P(s) \, dF_*\nu(s),
\end{equation}
for all polynomials $P$. Notice
\[
|F_*\mu|((0,1])=F_*|\mu|((0,1])=|\mu|([0,\infty))=M_{\mu}
\]
and similarly $|F_*\nu|((0,1])=M_{\nu}$. Therefore, each bounded and measurable (in Borel sense) function on $(0,1]$ is integrable with respect to the both measures $F_*\mu$ and $F_*\nu$. In view of
\[
\left|\int_{(0,1]}g(s)\, dF_*\mu(s)\right|\le M_{\mu}\, \|g\|_{\infty},\quad \left|\int_{(0,1]}g(s)\, dF_*\nu(s)\right|\le M_{\nu}\, \|g\|_{\infty},
\]
for all bounded measurable functions $g:(0,1]\rightarrow\mathbb{R}$, where $\|g\|_{\infty}=\sup\{|g(x)|\, :\, x\in (0,1]\}$, we conclude from Stone - Weierstrass theorem that
\begin{equation}\label{eq:SW}
\int_{(0,1]}g(s) \, dF_*\mu(s)=\int_{(0,1]}g(s) \, dF_*\nu(s),
\end{equation}
for all $g\in C[0,1]$. For small $\delta>0$ introduce a continuous, piecewise linear function $I_{\delta}:(0,1]\to\mathbb{R}$,
\[
I_{\delta}(t)=\left\{
\begin{array}{cc}
0,&t<a-\delta\\
\frac{t-(a-\delta)}\delta,&a-\delta\leq t\leq a\\
1,&a\leq t\leq b\\
\frac{b+\delta-t}\delta,&b\leq t\leq b+\delta\\
0,&b+\delta\leq t,
\end{array}
\right.
\]
where $[a,b]\subset (0,1]$. From \eqref{eq:SW}, we have
\begin{equation}\label{eq:delta}
\int_{(0,1]}I_{\delta}(s) \, dF_*\mu(s)=\int_{(0,1]}I_{\delta}(s) \, dF_*\nu(s).
\end{equation}
Taking into account that $I_{\delta}\to\chi_{[a,b]}$ pointwise as $\delta\to 0+$ (here $\chi$ denotes characteristic function) and $0\le I_{\delta}\le1$, one infers, applying Lebesgue dominant convergence theorem to integrals in \eqref{eq:delta}, that
$\int_{(0,1]}\chi_{[a,b]}(s) \, dF_*\mu(s)=\int_{(0,1]}\chi_{[a,b]}(s) \, dF_*\nu(s)$, or equivalently $F_*\mu([a,b])=F_*\nu([a,b])$, for all $[a,b]\subset (0,1]$. Following a similar procedure one can also deduce $F_*\mu((0,b])=F_*\nu((0,b])$, for all $(0,b]\subset(0,1]$. Therefore $F_*\mu(E)=F_*\nu(E)$ for all Borel sets $E\subset (0,1]$, which implies $F_*\mu=F_*\nu$. Finally, we obtain $\mu=\nu$, since $F$ is a homeomorphism.\ $\Box$
\begin{prop}\label{main}
Let $\varphi:[0,\infty)\rightarrow\mathbb{R}$ be a continuous function with the property
\begin{equation}\label{eq:cond}
\int_0^{\infty}|\varphi(t)|\, dt<\infty.
\end{equation}
If $f(x)=\int_0^{\infty}\varphi(t)\, e^{-xt}\, dt$ is completely monotonic, then $\varphi\ge0$.
\end{prop}
\begin{proof}
Since $f$ is completely monotonic, then according to Bernstein theorem, there exists a non-negative Borel measure $\mu$ on $[0,\infty)$ satisfying \eqref{eq:Bern} for all $x>0$. Due to \eqref{eq:cond}, we have
\[
\mu([0,\infty))=\int_{[0,\infty)}d\mu=f(0)=\int_0^{\infty}\varphi(t)\, dt<\infty,
\]
and consequently, $\mu$ is a finite measure. Again, thanks to \eqref{eq:cond}, we conclude that $\nu(E)=\int_E\varphi(t)\, dt$  is a Borel measure of bounded variation $|\nu|([0,\infty))=\int_0^{\infty}|\varphi(t)|\, dt<\infty$. 
Taking into account that
\[
\int_{[0,\infty)}e^{-xt}\, d\nu(t)=\int_0^{\infty}\varphi(t)\, e^{-xt}\, dt=f(x)=\int_{[0,\infty)}e^{-xt}\, d\nu(t),
\]
for all $x\ge0$, we see that the assumptions of Theorem \ref{moment} are fulfilled. Therefore, $\mu=\nu$. This implies
\[
\int_a^b \varphi(t)\, dt=\nu([a,b])=\mu([a,b])\ge0,
\]
for all $[a,b]\subset[0,\infty)$. However, $\varphi$ is continuous, whence $\varphi\ge0$. 
\end{proof}
\section{Applications}
We apply the results from the previous section with the aim to answer two questions stated in \cite{QiAgar} on page 34, which concern complete monotonicity of functions  $\psi'(x+1)-\sinh\frac1{x+1}$ and $\frac12\sinh\frac2x - \psi'(x+1)$. We will actually prove slightly more general assertions. Here $\psi(x)=\Gamma'(x)/\Gamma(x)$ is the digamma function.
\begin{prop}
For all $m>0$ the function $f(x)=\psi'(x+1)-\frac{1}{m}\sinh\frac{m}{x+1}$ is not completely monotonic.
\end{prop}
\begin{proof}
We employ the following representations
\begin{equation}\label{eq:repr}
\displaystyle \frac1{x^n}=\frac1{(n-1)!}\int_0^{\infty}t^{n-1}\, e^{-xt}\, dt, \quad \psi'(x)=\int_0^{\infty}\frac{t}{1-e^{-t}}\, e^{-xt}\, dt,
\end{equation}
for all $x>0$ and $n\in\mathbb{N}$. The latter one is due to S.\, Ramanujan (see \cite[p.\ 374]{Berndt}). From
\[
\displaystyle \frac{1}{m}\sinh\frac{m}{x+1}=\sum_{n=0}^{\infty}\frac{m^{2n}}{(2n+1)!}\,\frac1{(x+1)^{2n+1}},
\]
we conclude that
\begin{align*}
&\psi'(x+1)-\frac{1}{m}\sinh\frac{m}{x+1}=\\
&=\int_0^{\infty}\left(\frac{t}{1-e^{-t}}-\sum_{n=0}^{\infty}\frac{m^{2n}{t^{2n}}}{(2n)!(2n+1)!}\right)e^{-(x+1)t}\, dt \\
&= \int_0^{\infty}\varphi(t)\, e^{-xt}\, dt,&
\end{align*}
where $\varphi(t)=\left(\frac{t}{1-e^{-t}}-\sum_{n=0}^{\infty}\frac{m^{2n}{t^{2n}}}{(2n)!(2n+1)!}\right)e^{-t}$. Owing to $\frac{t}{1-e^{-t}}\sim t$ as $t\to \infty$ and $\sum_{n=0}^{\infty}\frac{m^{2n}{t^{2n}}}{(2n)!(2n+1)!}\ge \frac{m^2t^2}{4!\, 5!}$, one obtains that $\varphi$ is negative for large $t$. It is easy to see that $\int_0^{\infty}|\varphi(t)|\, dt<\infty$ and using Proposition \ref{main}, we infer that $f$ is not completely monotonic.
\end{proof}
\begin{prop}
For all $m>0$ function the $f(x)=\frac{1}{m}\sinh \frac{m}{x}-\psi'(x+1)$ is completely monotonic.
\end{prop}
\begin{proof}
Using \eqref{eq:repr}, we have 
\[
\frac{1}{m}\sinh \frac{m}{x}=\sum_{n=0}^\infty \frac{m^{2n}}{(2n+1)!\, x^{2n+1}}=
\int_0^\infty\sum_{n=0}^\infty\frac{m^{2n}t^{2n}}{(2n)!\,(2n+1)!}\, e^{-tx}\, dt,
\]
and
\[
\frac{1}{m}\sinh\frac{m}{x}-\psi'(x+1)=\int_0^\infty\left(\sum_{n=0}^\infty\frac{m^{2n}t^{2n}}{(2n)!(2n+1)!}-\frac{te^{-t}}{1-e^{-t}}\right)e^{-xt}\, dt,
\]
for all $x>0$. Hence
\begin{equation}\label{eq:final}
f(x)=\int_0^{\infty}\varphi(t)\, e^{-xt}\, dt,
\end{equation}
where $\displaystyle \varphi(t)=\sum_{n=0}^\infty\frac{m^{2n}t^{2n}}{(2n)!(2n+1)!}-\frac{te^{-t}}{1-e^{-t}}$. Further, it is
\begin{align*}
\varphi(t)(e^t-1)=(e^t-1)\sum_{n=0}^\infty\frac{m^{2n}t^{2n}}{(2n)!(2n+1)!}-t\ge t\cdot 1-t=0,
\end{align*}
for all $t\ge0$. Consequently, $\varphi\ge0$ on $[0,\infty)$ and by \eqref{eq:final} one concludes the proof.
\end{proof}


\begin{thebibliography}{1} 
%
\bibitem{BergForst} C.\, Berg, G.\, Forst, \emph{Potential Theory on Locally Compact Abelian Groups}, Springer-Verlag, Berlin Heidelberg New York, 1975.
%
\bibitem{Berndt} B.\, C.\, Berndt, \textit{Ramanujan's Notebook. Part IV}, Springer, New York, 1994.
%
\bibitem{Durrett} R.\, Durrett, \textit{Probability. Theory and Examples}, Cambridge University Press, 2010.
%
\bibitem{QiAgar}
F.~Qi and R.~Agarwal, \textit{On complete monotonicity for several classes of
  functions related to ratios of gamma functions}, J Inequal Appl  (2019),
  1--42.
%
\bibitem{Rudin} W.\, Rudin, \textit{Real and Complex Analysis}, McGraw-Hill, 1987.
%

\end{thebibliography}
\end{document}